\documentclass[11pt,reqno]{amsart}
\usepackage[T1]{fontenc}
\usepackage{lmodern}
\usepackage[a4paper,textwidth=6.15in,textheight=8.95in,centering]{geometry}
\usepackage{amsmath,amssymb,mathtools}
\usepackage{microtype,booktabs,array,needspace}
\usepackage[hidelinks,pdfencoding=auto]{hyperref}
\newtheorem{theorem}{Theorem}[section]
\newtheorem{lemma}[theorem]{Lemma}
\newtheorem{proposition}[theorem]{Proposition}
\newtheorem{corollary}[theorem]{Corollary}
\theoremstyle{remark}

\newcommand{\AI}{\mathsf{AI}}
\newcommand{\CAI}{\mathsf{CAI}}
\newcommand{\A}{\mathcal A_4}
\newcommand{\T}{\mathcal T}
\newcommand{\Ncal}{\mathcal N}
\newcommand{\F}{\mathbb F}
\newcommand{\eps}{\varepsilon}
\newcommand{\V}{\mathsf V}
\newcommand{\bu}{\mathbf u}
\newcommand{\bv}{\mathbf v}

\newcommand{\bs}{\mathbf s}
\newcommand{\bt}{\mathbf t}
\newcommand{\bfam}{\mathbf f}
\newcommand{\bK}{\mathbf k}

\DeclareMathOperator{\Aff}{Aff}
\numberwithin{equation}{section}
\allowdisplaybreaks[1]
\title[The joint variety of four-element ai-semirings]{The variety generated by all additively idempotent semirings of order four}

\author{Mengya Yue}
\address{School of Mathematics, Northwest University, Xi'an 710127, Shaanxi, P.R. China}
\email{myayue@yeah.net}
\thanks{Mengya Yue: ORCID 0009-0002-3939-947X.}

\author{Xiaolei Shao}
\address{School of Mathematics, Northwest University, Xi'an 710127, Shaanxi, P.R. China}
\email{xiaoleishao@yeah.net}

\date{}
\hypersetup{pdftitle={The variety generated by all additively idempotent semirings of order four},pdfauthor={Mengya Yue and Xiaolei Shao},pdfsubject={A structural proof of nonfinite basability}}
\begin{document}
\begin{abstract}
We prove that the variety generated by all four-element additively idempotent
semirings has no finite basis for its identities. 
\end{abstract}
\maketitle

\section{Introduction}\label{sec:intro}

An \emph{additively idempotent semiring}, or \emph{ai-semiring}, is an algebra
$(S,+,\cdot)$ in which $(S,+)$ is a commutative idempotent semigroup,
$(S,\cdot)$ is a semigroup, and
\[
x(y+z)\approx xy+xz,\qquad (x+y)z\approx xz+yz.
\]
All algebras considered here have signature $(+ ,\cdot)$, without constants.
An ai-semiring is called \emph{commutative} if its multiplication is
commutative. We denote the varieties of all ai-semirings and all commutative
ai-semirings by $\AI$ and $\CAI$, respectively. The additive order is given by
\[
a\leq b\quad\Longleftrightarrow\quad a+b=b.
\]
Both operations preserve this order.

For a class $\mathcal K$ of algebras, let $\V(\mathcal K)$ denote the variety
generated by $\mathcal K$. A variety is \emph{finitely based} if it can be
defined by finitely many identities, and \emph{nonfinitely based} otherwise.
The finite basis property of an algebra means the corresponding property
of its generated variety. We use the standard results of equational logic;
see Burris and Sankappanavar~\cite{BS}.

The finite basis problem for ai-semirings of small order has received
particular attention. Shao and Ren~\cite{SR} proved that the variety
generated by the two-element ai-semirings is hereditarily finitely based.
The work of Zhao et al.~\cite{ZRCSD} and Jackson, Ren, and Zhao~\cite{JRZ}
settled the problem for the individual three-element ai-semirings: among
the $61$ isomorphism types, only $S_7$ is nonfinitely based. For order four,
Ren et al.~\cite{RLLC, RLYC} and Yue et al.~\cite{YRZS} studied classes
determined by the additive semilattice. Further nonfinite basis results
were obtained for $S_{(4,124)}$ by Yue and Ren~\cite{YR124}, and for
$S_{(4,545)}$ and $S_{(4,634)}$ by Yue, Ren, and Gao~\cite{YRG}.

The joint variety generated by all algebras of a given order presents a
different question from the classification of its individual generators.
In particular, nonfinite basability does not follow merely from the
inclusion of a nonfinitely based subvariety. Wang, Ju, and Wang~\cite{WJW}
proved that the joint variety of three-element ai-semirings is nonfinitely
based, and obtained the same conclusion for three-element semirings under
a more general convention on addition. Here we consider
\[
\A=\V\{S\in\AI:|S|=4\}.
\]
Our main result is the following.

\begin{theorem}\label{thm:main}
The variety $\A$ is nonfinitely based.
\end{theorem}

Syntactic sufficient conditions for nonfinite basability were developed
in~\cite{YR124,YRG,LRY}. We establish a condition requiring a variety to
contain four specified commutative ai-semirings and to satisfy certain
cyclic identities. The auxiliary semirings restrict the multiplicities and
contents of words, and impose affine conditions on their exponent vectors.
Together these restrictions determine an equivalence class containing only
two terms. We then show that an identity used to pass between those terms
must involve an arbitrarily large number of variables.

The construction is based on the isolated-class argument of~\cite{WJW}.
The construction for four-element generators uses squared variables,
three consecutive variables in each local word, and flat extensions of
cyclic groups of both orders two and three. Repeated variables require an additional argument:
under a semiring substitution, their different occurrences may select
different words from the same image. The proof below treats these choices
without restricting substitutions to words.

Section~\ref{sec:prelim} gives the notation and the auxiliary semirings.
In Section~\ref{sec:criterion} we prove the sufficient condition.
Section~\ref{sec:application} verifies its hypotheses for $\A$ and proves
Theorem~\ref{thm:main}. The structural argument distinguishes a square
that is not idempotent from a semiring in which every square is idempotent.
In the latter case the decisive distinction is whether the multiplicative
idempotents form a subsemiring. Their failure to do so leaves at most four
configurations, each determined directly from the semiring axioms. Outside $G_3$, we obtain a stronger absorption property in which
the left coefficients of the local words are arbitrary. The congruence
condition on the cycle length is needed only for $G_3$.

\section{Preliminaries}\label{sec:prelim}

Fix a countably infinite alphabet $X$. Write $X^+$ and $X^*$ for the free
semigroup and free monoid on $X$, and $X_c^+$ and $X_c^*$ for their
commutative counterparts. The empty word is denoted by $\eps$. By
distributivity and additive idempotence, an ai-semiring term is represented
by a finite nonempty subset of $X^+$. We write its elements as summands:
\[
\bu=u_1+\cdots+u_m=\{u_1,\ldots,u_m\}.
\]
The set $P_f(X^+)$ of such terms is the free ai-semiring on $X$, with union
as addition and setwise concatenation as multiplication; see
\cite[Section~1]{RLLC}. Thus terms will be denoted by bold letters and
individual words by ordinary letters. In particular, $w\in\bu$ means that
$w$ is a word of $\bu$.

For a word $w$, let $c(w)$ be its content, let $\ell(w)$ be its length,
and let $m(x,w)$ be the multiplicity of $x$ in $w$. For a term $\bu$, put
$c(\bu)=\bigcup_{w\in\bu}c(w)$. A substitution is an endomorphism of
$P_f(X^+)$; its values on variables are arbitrary nonempty terms. For
commutative ai-semirings the same notation is used with $X_c^+$ in place
of $X^+$. The natural homomorphism
\[
P_f(X^+)\longrightarrow P_f(X_c^+),\qquad \bu\longmapsto\overline{\bu},
\]
forgets the order of letters in each word. Equality of terms, denoted by
$=$, is equality of the represented sets; an identity is denoted by
$\approx$.

We write $\bv\preceq\bu$ for the identity $\bu+\bv\approx\bu$.
This notation is distinct from the literal inclusion $\bv\subseteq\bu$.
If $\bu=u_1+\cdots+u_m$ and $\bv=v_1+\cdots+v_t$, then
$\bu\approx\bv$ is equivalent, within $\AI$, to the identities
\[
u_i\preceq\bv\quad(1\leq i\leq m),\qquad
v_j\preceq\bu\quad(1\leq j\leq t).
\]
Indeed, these inequalities give both $\bu+\bv\approx\bu$ and
$\bu+\bv\approx\bv$. This replacement does not introduce variables.
It therefore suffices to consider inequalities $r\preceq\bs$ with $r$ a
word and $\bs$ a term.

\Needspace{8\baselineskip}
We shall use the following form of equational deduction. Derivability in
its statement is relative to the ai-semiring identities.

\begin{lemma}\label{lem:deduction}
An identity is a consequence of a set $\Sigma$ of identities if and only if
its sides can be joined by a finite sequence of contextual substitution
instances of members of $\Sigma$, used in either direction. For an
inequality $r\preceq\bs$, such a step has the form
\begin{equation}\label{eq:deduction}
\mathbf l+\mathbf p\varphi(\bs)\mathbf q
\quad\longleftrightarrow\quad
\mathbf l+\mathbf p\varphi(\bs)\mathbf q+
\mathbf p\varphi(r)\mathbf q.
\end{equation}
The factors $\mathbf p,\mathbf q$ may be absent, and $\mathbf l$ may be
absent. In $P_f(X_c^+)$ the two multiplicative contexts combine into a
single context $\bK$.
\end{lemma}
\begin{proof}
The description by contextual instances is the usual completeness theorem
for equational logic; in the ai-semiring setting it is recorded, for
example, in~\cite[Lemma~1.1]{RLLC}. The displayed form follows by expanding
a context with one distinguished occurrence of its argument. Starting
with that occurrence, addition contributes to $\mathbf l$, and left or
right multiplication contributes to the corresponding multiplicative
context, with distributivity at each stage. Substitution of $\varphi(\bs)$
and $\varphi(\bs)+\varphi(r)$ gives~\eqref{eq:deduction}. In the
commutative case the contexts multiply to $\bK=\mathbf p\mathbf q$.
\end{proof}

An absent factor is represented by $\eps$, and an absent additive context
by the empty set. These conventions concern contexts only: neither is an
allowed value of a variable under a substitution.

We next specify four commutative ai-semirings. Let
\[
H=\{1,a,a^2,\infty\},\qquad 1<a<a^2<\infty,
\]
with maximum as addition. Set $a^0=1$, $a^3=\infty$, and define
\[
a^ia^j=a^{\min\{i+j,3\}}\qquad(0\leq i,j\leq3).
\]
Let $D$ be the two-element distributive lattice, with join as addition and
meet as multiplication. For $p\in\{2,3\}$, let $G_p$ be the flat extension
of the cyclic group $C_p=\langle g\rangle$. Thus
\[
G_p=C_p\cup\{\infty\},\qquad
u+v=\begin{cases}u,&u=v,\\ \infty,&u\ne v,\end{cases}
\]
and multiplication extends the group multiplication with $\infty$
absorbing. The element names do not add constants to the signature.

The multiplication of $H$ is associative by truncated addition of
exponents and is order preserving on the chain. Hence it preserves joins
in each argument. For $G_p$, multiplication by a group element is a
permutation of $C_p$ fixing $\infty$, so it preserves the flat addition.
Multiplication by $\infty$ is constant. These observations verify the
ai-semiring axioms. The Cayley tables of $H$ and $G_3$ are given in
Table~\ref{tab:aux}.

\begin{table}[htbp]
\caption{The auxiliary semirings $H$ and $G_3$.}\label{tab:aux}
\centering
\renewcommand{\arraystretch}{1.06}
\setlength{\tabcolsep}{6pt}
$\begin{array}{c|cccc}
+&1&a&a^2&\infty\\\hline
1&1&a&a^2&\infty\\
a&a&a&a^2&\infty\\
a^2&a^2&a^2&a^2&\infty\\
\infty&\infty&\infty&\infty&\infty
\end{array}\qquad
\begin{array}{c|cccc}
\cdot&1&a&a^2&\infty\\\hline
1&1&a&a^2&\infty\\
a&a&a^2&\infty&\infty\\
a^2&a^2&\infty&\infty&\infty\\
\infty&\infty&\infty&\infty&\infty
\end{array}$
\par\medskip
$\begin{array}{c|cccc}
+&1&g&g^2&\infty\\\hline
1&1&\infty&\infty&\infty\\
g&\infty&g&\infty&\infty\\
g^2&\infty&\infty&g^2&\infty\\
\infty&\infty&\infty&\infty&\infty
\end{array}\qquad
\begin{array}{c|cccc}
\cdot&1&g&g^2&\infty\\\hline
1&1&g&g^2&\infty\\
g&g&g^2&1&\infty\\
g^2&g^2&1&g&\infty\\
\infty&\infty&\infty&\infty&\infty
\end{array}$
\end{table}

Put
\begin{equation}\label{eq:T}
\T=\V(H,D,G_2,G_3).
\end{equation}

\begin{lemma}\label{lem:inclusion}
Every ai-semiring of order at most four belongs to $\A$. In particular,
$\T\subseteq\A\cap\CAI$.
\end{lemma}
\begin{proof}
An ai-semiring $S$ embeds in an ai-semiring with one more element by
adjoining $\omega$ and making it absorbing for both operations. The old
operations are unchanged. An associativity or distributivity instance
involving $\omega$ has value $\omega$ on both sides; the remaining
instances hold in $S$. Addition remains commutative and idempotent.
Iteration gives an embedding into a four-element ai-semiring whenever
$|S|<4$. Apply this to the four algebras in~\eqref{eq:T}.
\end{proof}

For a commutative word $w$ and a finite set $Y$ containing $c(w)$, denote
its exponent vector by
\[
\nu(w)=(m(x,w))_{x\in Y}\in\mathbb Z_{\geq0}^{Y}.
\]
Its reduction modulo a prime $p$ is denoted by $\nu_p(w)$. An affine
combination of vectors over $\F_p$ is a linear combination whose
coefficients have sum one.

\begin{lemma}\label{lem:characterization}
Let $r$ be a nonempty commutative word and
$\bu=u_0+\cdots+u_t$ a commutative term.
\begin{enumerate}
\item The inequality $r\preceq\bu$ holds in $D$ if and only if
$c(u_i)\subseteq c(r)$ for some $i$.
\item For $p\in\{2,3\}$, the inequality $r\preceq\bu$ holds in $G_p$
if and only if $c(r)\subseteq c(\bu)$ and
\begin{equation}\label{eq:affine}
\nu_p(r)\in\Aff_{\F_p}\{\nu_p(u_0),\ldots,\nu_p(u_t)\}.
\end{equation}
\end{enumerate}
\end{lemma}
\begin{proof}
In $D$, a word has value one precisely when all its variables have value
one. The stated content inclusion is therefore sufficient. For necessity,
assign one to the variables of $r$ and zero to the others.

For $G_p$, a variable of $r$ outside $c(\bu)$ can be assigned $\infty$
while all other variables receive $1$, disproving the inequality. Thus
content inclusion is necessary. Choose a finite alphabet $Y$ containing
both contents. A valuation in the group part is specified by
$x\mapsto g^{\alpha_x}$, where $\alpha\in\F_p^Y$, and the value of $w$
is $g^{\nu_p(w)\cdot\alpha}$.

Let $M$ be the matrix with rows $\nu_p(u_i)-\nu_p(u_0)$,
$1\leq i\leq t$. If $\alpha\in\ker M$, all words of $\bu$ have the
same group value. Distinct group elements are incomparable in the additive
order, so validity forces
\[
(\nu_p(r)-\nu_p(u_0))\cdot\alpha=0\qquad(\alpha\in\ker M).
\]
The orthogonal complement of $\ker M$ is the row space of $M$, by
rank--nullity. This yields~\eqref{eq:affine}.

Conversely, assume the two conditions. Under a valuation for which
$\bu=\infty$, absorption is automatic. Otherwise all words $u_i$ take a
common group value $g^b$, and every variable of $\bu$, hence of $r$,
has a group value. Write $\nu_p(r)=\sum_i\lambda_i\nu_p(u_i)$ with
$\sum_i\lambda_i=1$. The exponent of the value of $r$ is
$\sum_i\lambda_i b=b$. Thus $r$ has the same value as $\bu$.
\end{proof}

\section{A sufficient condition for the nonfinite basis property}
\label{sec:criterion}

Let
\[
\Ncal=\{n\in\mathbb N:n\geq7,\ n\equiv1\pmod3\}.
\]
For $n\in\Ncal$, use pairwise distinct variables
$x_1,\ldots,x_n,y_1,\ldots,y_n,z$, with the subscripts on the $y$-variables
read modulo $n$. Define
\begin{align}
p_n&=x_1^2x_2^2\cdots x_n^2,& a_n&=x_1^2z,& q_n&=p_nz,
\label{eq:words}\\
\bfam_n&=p_n+a_n+\sum_{i=1}^n y_i y_{i+1}y_{i+2}x_i^2z.
\label{eq:family}
\end{align}
Here the products have the displayed order; no commutativity is assumed.
Let $\sigma_n$ denote the identity
\begin{equation}\label{eq:sigma}
q_n\preceq\bfam_n.
\end{equation}

\begin{theorem}\label{thm:criterion}
Let $\mathcal V$ be an ai-semiring variety containing $H,D,G_2$, and
$G_3$. If $\mathcal V$ satisfies $\sigma_n$ for infinitely many
$n\in\Ncal$, then $\mathcal V$ has no equational basis whose identities
involve a bounded number of variables. In particular, $\mathcal V$ is
nonfinitely based.
\end{theorem}

To prove the theorem we work with commutative terms. In this setting the
symbols $p_n,a_n,q_n$ denote the commutative images of the words
in~\eqref{eq:words}. Put
\begin{equation}\label{eq:u}
b_i=x_i^2y_i y_{i+1}y_{i+2}z,\qquad
\bu_n=p_n+a_n+\sum_{i=1}^n b_i=\overline{\bfam_n}.
\end{equation}
The words in $\bu_n+q_n$ are pairwise distinct. We shall characterize the
terms equivalent to $\bu_n$ in $\T$ and then examine substitutions into
these terms.

\begin{lemma}\label{lem:validT}
The variety $\T$ satisfies $q_n\preceq\bu_n$ for every $n\in\Ncal$.
\end{lemma}
\begin{proof}
In $D$, $q_n=p_nz\leq p_n$. For a valuation in $H$, if $z=1$ then
$q_n=p_n$, and if all $x_i=1$ then $q_n=a_n=z$. Otherwise some
$x_i\geq a$ and $z\geq a$, whence $x_i^2z=\infty$ and $b_i=\infty$.
This proves absorption in $H$.

Consider a valuation in $G_p$. Unless $\bu_n=\infty$, all its words
have a common group value $c$. Every variable occurs in $\bu_n$, so all
variables then take group values. When $p=2$, the squares in $p_n$ give
$p_n=1$, hence $c=1$; the equality $a_n=c$ gives $z=1$, and $q_n=c$.
When $p=3$, multiplying the equalities $b_i=c$ gives
\[
c^n=\prod_{i=1}^n b_i
=p_n\left(\prod_{i=1}^n y_i^3\right)z^n=p_nz^n.
\]
Since $n\equiv1\pmod3$ and $p_n=c$, it follows that $c=cz$. Cancellation
in the group yields $z=1$, and again $q_n=c$. Thus the inequality holds
in every generator of $\T$.
\end{proof}

\begin{lemma}\label{lem:word}
For $n\in\Ncal$ and $r\in X_c^+$,
\[
\T\models r\preceq\bu_n
\quad\Longleftrightarrow\quad
r\in\bu_n+q_n.
\]
\end{lemma}
\begin{proof}
The reverse implication follows from Lemma~\ref{lem:validT}. Suppose that
$\T\models r\preceq\bu_n$. Assigning $a$ to a single variable and $1$
to the others in $H$ gives
\begin{equation}\label{eq:degrees}
\begin{gathered}
c(r)\subseteq\{x_1,\ldots,x_n,y_1,\ldots,y_n,z\},\\
m(x_i,r)\leq2,\qquad m(y_j,r)\leq1,\qquad m(z,r)\leq1.
\end{gathered}
\end{equation}
Indeed, $\bu_n$ has value $a^2$ when the distinguished variable is $x_i$,
value $a$ when it is $y_j$ or $z$, and value $1$ for a variable outside
its content. Next, assign the nonidentity group element of $G_2$ to a
single $x_i$, and $1$ to all other variables. Every word of $\bu_n$
has value $1$, so $m(x_i,r)$ is even. Hence
\begin{equation}\label{eq:even}
m(x_i,r)\in\{0,2\}.
\end{equation}

We also have the following property:
\begin{equation}\label{eq:pairs}
\text{each pair of distinct variables in $c(r)$ occurs together in a word of $\bu_n$.}
\end{equation}
To see this, the pairs not occurring together are of the form $x_i,y_j$
with $y_j\notin\{y_i,y_{i+1},y_{i+2}\}$, or are pairs of $y$-variables
not belonging to a common triple of consecutive $y$-variables. In the
first case assign $x_i=a$ and $y_j=a^2$ in $H$, with all other variables
$1$. The value of $\bu_n$ is at most $a^2$, whereas $r=\infty$ by
\eqref{eq:even}. In the second case assign $a^2$ to both $y$-variables
and $1$ to the others, obtaining the same contradiction.

By Lemma~\ref{lem:characterization}(1), some word of $\bu_n$ has content
contained in $c(r)$. We distinguish three cases.

\emph{Case 1. $c(p_n)\subseteq c(r)$.} Every $y_j$ occurs together with
only three $x$-variables in $\bu_n$. Since $n\geq7$, property
\eqref{eq:pairs} excludes all $y$-variables from $r$. Equations
\eqref{eq:degrees} and~\eqref{eq:even} now give $r=p_n$ or $r=q_n$.

\emph{Case 2. $c(b_i)\subseteq c(r)$ for some $i$.} The occurrence of
$x_i$ excludes further $y$-variables. An additional $x_j$ would have to
occur together with all of $y_i,y_{i+1},y_{i+2}$. Its three $y$-neighbours
would therefore coincide with this triple. Distinct cyclic triples have
distinct starting indices for $n\geq7$, so $j=i$. Thus $c(r)=c(b_i)$,
and the degree restrictions imply $r=b_i$.

\emph{Case 3. $c(a_n)\subseteq c(r)$.} Write
\[
m(x_i,r)=2\delta_i,\quad \delta_i\in\{0,1\},\quad\delta_1=1,
\qquad m(y_j,r)=\eta_j\in\{0,1\}.
\]
We have $m(z,r)=1$, and~\eqref{eq:pairs} gives $\eta_j=0$ for
$4\leq j\leq n$. By Lemma~\ref{lem:characterization}(2), there exist
$\lambda_P,\lambda_A,\lambda_1,\ldots,\lambda_n\in\F_3$ such that
\begin{align*}
\nu_3(r)&=\lambda_P\nu_3(p_n)+\lambda_A\nu_3(a_n)
              +\sum_{i=1}^n\lambda_i\nu_3(b_i),\\
1&=\lambda_P+\lambda_A+\sum_{i=1}^n\lambda_i.
\end{align*}
All coefficient equations in this case are over $\F_3$. The
$z$-coordinate gives $1=\lambda_A+\sum_i\lambda_i$, so
$\lambda_P=0$. The $x_i$- and $y_j$-coordinates then give
\begin{equation}\label{eq:lambda}
\lambda_i=\delta_i-\lambda_A\mathbf1_{\{i=1\}},\qquad
\eta_j=\lambda_j+\lambda_{j-1}+\lambda_{j-2}.
\end{equation}
Consequently,
\[
\delta_j+\delta_{j-1}+\delta_{j-2}=0\quad\text{in }\F_3
\qquad(4\leq j\leq n).
\]
Each $\delta_i$ is the integer zero or one, so three such entries have
sum zero modulo three only when they are all equal. The overlapping
triples imply
\[
\delta_2=\delta_3=\cdots=\delta_n.
\]
If these entries are one, Case~1 gives $r=q_n$. If they are zero,
\eqref{eq:lambda} gives $\lambda_i=0$ for $i\ne1$ and
\[
\eta_1=\eta_2=\eta_3=1-\lambda_A.
\]
These three degrees are all zero or all one. Hence $r=a_n$ or $r=b_1$.
The three cases exhaust the content inclusion supplied by $D$.
\end{proof}

\begin{lemma}\label{lem:class}
For every $n\in\Ncal$ and every commutative term $\bv$,
\[
\T\models\bv\approx\bu_n
\quad\Longleftrightarrow\quad
\bv=\bu_n\ \text{or}\ \bv=\bu_n+q_n.
\]
\end{lemma}
\begin{proof}
Assume $\T\models\bv\approx\bu_n$. Every word of $\bv$ is absorbed
by $\bu_n$, so Lemma~\ref{lem:word} gives
$\bv\subseteq\bu_n+q_n$. We show that no word of $\bu_n$ can be
omitted from $\bv$.

For $p_n$, evaluate in $D$ with all $x_i=1$ and all $y_i=z=0$.
Then $p_n=1$ and every other word of $\bu_n+q_n$ has value zero.
For $a_n$, evaluate in $D$ with $x_1=z=1$ and all other variables zero.
Again the selected word has value one and all the others have value zero.
Finally, for $b_i$, evaluate in $H$ with $y_i=y_{i+2}=a$ and every other
variable $1$. Only the triple starting at $i$ contains both of these
$y$-variables, since $n\geq7$. Thus $b_i=a^2$, while $b_j\leq a$
for $j\ne i$ and $p_n=a_n=q_n=1$.

Each valuation separates the selected word from the sum of all remaining
words. Since every word of $\bu_n$ must be absorbed by $\bv$, all must
occur in $\bv$. Therefore
$\bu_n\subseteq\bv\subseteq\bu_n+q_n$, leaving exactly the two stated
terms. The converse is Lemma~\ref{lem:validT}.
\end{proof}

The next lemma is the main restriction on substitutions. The assumption
on $h\varphi(\bs)$ is literal inclusion of sets of commutative words.

\begin{lemma}\label{lem:substitution}
Let $n\in\Ncal$, let $\bs\in P_f(X_c^+)$ and $r\in X_c^+$, and
suppose that $\T\models r\preceq\bs$. If a substitution $\varphi$ and
a word $h\in X_c^*$ satisfy
\begin{equation}\label{eq:hyp}
h\varphi(\bs)\subseteq\bu_n,\qquad q_n\in h\varphi(r),
\end{equation}
then $|c(r)|\geq n+1$.
\end{lemma}
\begin{proof}
Each word of $h\varphi(r)$ is absorbed by $\bu_n$ in $\T$, because
$r\preceq\bs$ and $h\varphi(\bs)\subseteq\bu_n$. By
Lemma~\ref{lem:word},
\begin{equation}\label{eq:image}
h\varphi(r)\subseteq\bu_n+q_n.
\end{equation}
Also $c(r)\subseteq c(\bs)$: otherwise, in $H$, assigning $a$ to a
variable of $c(r)\setminus c(\bs)$ and $1$ to all others contradicts
$r\preceq\bs$. We use $v$ for the variables of $r$ and $\bs$.

\emph{Claim 1.} There are nonempty words $w_v\in\varphi(v)$ such that
\begin{equation}\label{eq:factor}
q_n=h\prod_{v\in c(r)}w_v^{m_v},\qquad
m_v=m(v,r)\in\{1,2\}.
\end{equation}

First, $m_v\leq2$. If $m_v\geq3$, choose the same nonempty word of
$\varphi(v)$ at all occurrences of $v$ in an expansion of
$h\varphi(r)$. Some letter would then have multiplicity at least three,
contrary to~\eqref{eq:image}.

Choose an expansion giving $q_n$. All selected words and $h$ involve
only $x_1,\ldots,x_n,z$. Suppose that the two occurrences of a variable
of multiplicity two select $w$ and $w'$. Let $k$ be the product of the
remaining selected factors, including $h$. Then $kww'=q_n$. Both
$kw^2$ and $k(w')^2$ also occur in $h\varphi(r)$. Neither $w$ nor $w'$
can contain $z$, for its diagonal product would contain $z^2$. Thus $k$
contains $z$. The diagonal products contain no $y$-variables, so by
\eqref{eq:image} each is $a_n$ or $q_n$. In integer exponent vectors,
\[
2\nu(q_n)=\nu(kw^2)+\nu(k(w')^2).
\]
Since $a_n$ is a proper divisor of $q_n$, both diagonal products must
be $q_n$. It follows that $2\nu(w)=2\nu(w')$, whence $w=w'$.
Applying this to each repeated variable proves the claim. All cancellations
here are in the free commutative monoid.

\emph{Claim 2.} The word $h$ does not contain $z$.

Suppose that $z\in c(h)$. Each $w_v$ in~\eqref{eq:factor} is then a
nonempty word in the $x$-variables. Since $c(r)\subseteq c(\bs)$, for
each $v\in c(r)$ we can choose an occurrence of $v$ in a word of $\bs$
and select $w_v$ at that occurrence. The resulting expansion, after
multiplication by $h$, belongs to $\bu_n$ and is divisible by $hw_v$.
Every word of $\bu_n$ containing $z$ has total $x$-degree two and
contains only one distinct $x_i$. Consequently, $h$ contains at most one
distinct $x_i$, and the same is true of each $w_v$.

For any $t\in\bs$, select $w_v$ at every occurrence of each
$v\in c(r)$, and arbitrary words at the other occurrences. Every selected
occurrence from $c(r)$ contributes at least one $x$-letter. Since the
expanded word, multiplied by $h$, is a word of $\bu_n$ containing $z$,
\begin{equation}\label{eq:sumdegree}
\sum_{v\in c(r)}m(v,t)\leq2\qquad(t\in\bs).
\end{equation}
Now evaluate the formal variables in $H$, assigning $a$ to all variables
of $r$ and $1$ to the others. Equation~\eqref{eq:sumdegree} gives
$\bs\leq a^2$. The inequality $r\preceq\bs$ implies $\ell(r)\leq2$.
The factorization~\eqref{eq:factor} then supplies at most
$1+\ell(r)\leq3$ distinct $x$-variables, contrary to the $n\geq7$
distinct $x$-variables of $q_n$. This proves Claim~2.

There is consequently a unique $v_0\in c(r)$ whose selected word
contains $z$, and
\begin{equation}\label{eq:v0}
m_{v_0}=1,\qquad m(z,w_{v_0})=1.
\end{equation}
Every other selected word is a nonempty word in the $x$-variables.
For every $t\in\bs$, the first condition of~\eqref{eq:hyp} implies
\begin{equation}\label{eq:sdegree}
m(v_0,t)\leq1,\qquad m(v,t)\leq2
\quad(v\in c(r)\setminus\{v_0\}).
\end{equation}
Indeed, two choices of $w_{v_0}$ would produce $z^2$, and three choices
of any other $w_v$ would produce a letter of multiplicity at least three.

There is at least one variable in $c(r)\setminus\{v_0\}$. Otherwise
$r=v_0$ and $hw_{v_0}=q_n$. An occurrence of $v_0$ in $\bs$ would give
a word of $\bu_n$ divisible by $q_n$, which is impossible.

\emph{Claim 3.} We have $h=\eps$, $w_{v_0}=z$, and each
$w_v$ with $v\ne v_0$ contains only one distinct $x$-variable.

Fix $v\in c(r)\setminus\{v_0\}$. Assign $v_0=a^2$, $v=a$ and all
other formal variables $1$ in $H$. Then $r=\infty$, so $\bs=\infty$.
By~\eqref{eq:sdegree}, this is possible only if some word of $\bs$
contains both $v_0$ and $v$. Select their words $w_{v_0}$ and $w_v$ in
an expansion. We obtain
\begin{equation}\label{eq:divides}
hw_{v_0}w_v\text{ divides a word of $\bu_n$ containing $z$.}
\end{equation}
Such a word contains only one distinct $x$-variable. If $hw_{v_0}$
contained $x_j$, equation~\eqref{eq:divides} would force every other
$w_v$ to involve only $x_j$. This contradicts~\eqref{eq:factor}.
Thus $hw_{v_0}$ contains no $x$-variable. Claim~2 and~\eqref{eq:v0}
now give $h=\eps$ and $w_{v_0}=z$. Equation~\eqref{eq:divides} also
shows that every other $w_v$ contains only one distinct $x$-variable,
proving Claim~3.

To supply all $x_1,\ldots,x_n$ in~\eqref{eq:factor}, at least $n$
distinct variables other than $v_0$ are therefore required. Hence
$|c(r)|\geq n+1$.
\end{proof}

\begin{proof}[Proof of Theorem~\ref{thm:criterion}]
Suppose that $\mathcal V$ has a basis $\Sigma$ with a variable bound $k$.
Choose $n\in\Ncal$ such that $n+1>k$ and $\mathcal V\models\sigma_n$.
The identity $\sigma_n$ is then a consequence of $\Sigma$.

Adjoin $xy\approx yx$ to the deduction system and pass to commutative
terms. This is legitimate for testing nonderivability: a deduction from
$\Sigma$ remains a deduction in the enlarged system. Moreover, all its
identities are valid in $\T$. Normalize each member of $\Sigma$ and
replace it by inequalities $r\preceq\bs$ as in Section~\ref{sec:prelim}.
The resulting set $\Sigma'$ has the same consequences modulo $\CAI$
and still uses at most $k$ variables in each identity.

By Lemma~\ref{lem:deduction}, there is a sequence of commutative terms
\[
\bu_n=\bt_0,\bt_1,\ldots,\bt_m=\bu_n+q_n
\]
whose steps are contextual instances of members of $\Sigma'$. Every
$\bt_i$ is equivalent to $\bu_n$ in $\T$. Lemma~\ref{lem:class} shows
that each is either $\bu_n$ or $\bu_n+q_n$.

Consider the first step that changes $\bu_n$. It must go from
$\bu_n$ to $\bu_n+q_n$. A contraction in~\eqref{eq:deduction} cannot
introduce a word, so the step has the form
\[
\mathbf l+\bK\varphi(\bs)
\ \longrightarrow\
\mathbf l+\bK\varphi(\bs)+\bK\varphi(r)
\]
for some $r\preceq\bs$ in $\Sigma'$. Since $q_n$ occurs only on the
right, choose $h\in\bK$ such that $q_n\in h\varphi(r)$. If the
multiplicative context is absent, take $h=\eps$. The left side gives
$h\varphi(\bs)\subseteq\bu_n$. Lemma~\ref{lem:substitution} yields
$|c(r)|\geq n+1>k$, a contradiction.

Thus there is no equational basis with a bounded number of variables.
Every finite basis has such a bound, proving nonfinite basability.
\end{proof}

\section{A structural verification of the cyclic identities}
\label{sec:application}

It remains to prove that the identities $\sigma_n$ hold in every
four-element ai-semiring. The only case in which the cyclic arrangement
of the $y$-variables is needed is the flat extension $G_3$ of the group
of order three. For the other semirings, we prove a stronger absorption
property with arbitrary left coefficients.

For $n\geq2$ and elements $x_1,\ldots,x_n,z,t_1,\ldots,t_n$ of an
ai-semiring $S$, put
\[
s_i=x_i^2\quad(1\leq i\leq n),\qquad p=s_1s_2\cdots s_n.
\]
The stronger property is
\begin{equation}\label{eq:structural-absorption}
pz\leq p+s_1z+\sum_{i=1}^{n}t_i s_i z.
\end{equation}
Here and below the symbols in this inequality denote elements of $S$.
No commutativity of multiplication is assumed.

\begin{proposition}\label{prop:strong-absorption}
Every four-element ai-semiring not isomorphic to $G_3$ satisfies
\eqref{eq:structural-absorption} for every $n\geq2$ and every choice of
its indicated elements.
\end{proposition}

We prepare several elementary structural observations. Write
\[
E(S)=\{e\in S:e^2=e\}
\]
for the set of multiplicative idempotents.

\begin{lemma}\label{lem:small-powers}
Let $S$ be a four-element ai-semiring.
\begin{enumerate}
\item If its multiplicative reduct contains a subgroup of order three,
then $S$ is isomorphic to $G_3$.
\item If $S$ is not isomorphic to $G_3$, then $a^4\in E(S)$ for every
$a\in S$. If $a^2\notin E(S)$, then either
\[
\langle a\rangle=\{a,a^2,a^3\},\qquad a^4=a^3,
\]
or
\[
S=\langle a\rangle=\{a,a^2,a^3,a^4\},\qquad a^5=a^4,
\]
where the displayed powers in each set are distinct.
\end{enumerate}
\end{lemma}
\begin{proof}
For any finite multiplicative subgroup $K$ of $S$, let $b=\sum_{g\in K}g$.
Distributivity gives
\[
b^2=b,\qquad gb=bg=b\quad(g\in K).
\]
If $K$ is nontrivial, then $b\notin K$, since otherwise cancellation in
$K$ would force every element of $K$ to be its identity. In particular,
$S$ cannot itself be a nontrivial group.

Suppose that $|K|=3$. Then $S=K\cup\{b\}$, and $b$ is the greatest
additive element and a multiplicative zero. Distinct elements of $K$
are incomparable in the additive order. Indeed, an inequality $g\leq h$
inside $K$ gives $1\leq g^{-1}h$; multiplying repeatedly by $g^{-1}h$
gives a cycle of inequalities and hence $g=h$. Thus the join of any two
distinct elements of $K$ is $b$. This proves (1).

For (2), recall the elementary description of a finite cyclic semigroup.
There are an index $r\geq1$ and a period $d\geq1$ such that its distinct
powers are $a,a^2,\ldots,a^{r+d-1}$ and $a^{r+d}=a^r$.
The periodic part is a cyclic group of order $d$: choose a multiple $k$
of $d$ with $k\geq r$, and use $a^k$ as its identity. Here
$r+d-1\leq4$. Part (1) excludes $d=3$ when $S\not\cong G_3$, and
$d=4$ would make $S$ a group. If $d=2$ and $r=3$, then $S=\langle a\rangle$
has the subgroup $\{a^3,a^4\}$ and has only one idempotent, namely $a^4$.
Its subgroup join would be a different idempotent, by the first paragraph,
a contradiction. Consequently, either $d=1$ and $r\leq4$, or $d=2$ and
$r\leq2$. In both cases $a^4$ is idempotent. The condition that $a^2$
is not idempotent leaves precisely $d=1$ with $r=3$ or $r=4$.
\end{proof}

\begin{lemma}\label{lem:idempotent-closure}
Suppose that every square in a four-element ai-semiring $S$ is
idempotent. If $E(S)$ is not a subsemiring, then there are elements
$e,f,u,h$ such that
\begin{equation}\label{eq:exceptional-shape}
S=\{e,f,u,h\},\qquad E(S)=\{e,f,h\},\qquad
 e+f=u<h,\qquad u^2=h.
\end{equation}
After interchanging $e$ and $f$ if necessary, one has $ef=h$, while
$fe\in\{e,f,u,h\}$.
\end{lemma}
\begin{proof}
We first show that failure of multiplicative closure of $E(S)$ forces
failure of its additive closure. Suppose that $a,b\in E(S)$ but
$c=ab\notin E(S)$, and put $d=c^2\in E(S)$. We have $ac=c=cb$.
If $d=a$, then $ab=db=c^2b=c^2=a$, a contradiction. Similarly,
$d=b$ would give $ab=ad=ac^2=c^2=b$. Thus $a,b,c,d$ are four distinct
elements and constitute $S$. The equality
\[
d=c^2=a(ba)b
\]
excludes $ba\in\{a,b,c\}$, as each of these choices makes its right
side $c$. Hence $ba=d$. Associativity now gives
\[
ad=da=bd=db=d,
\]
and $d$ is a multiplicative zero.

The elements $a,b$ are incomparable. For example, $a\leq b$ would give
$c=ac\leq bc=d$ and $d=ca\leq cb=c$, a contradiction; the reverse
inequality is excluded in the same way. Therefore $a+b\in\{c,d\}$.
If $a+b=d$, then $(a+b)b=d$ gives $c+b=d$, whereas
\[
a(c+b)=ac+ab=c\ne d=ad.
\]
It follows that $a+b=c\notin E(S)$.

Now suppose that additive closure fails, and choose $e,f\in E(S)$ with
$u=e+f\notin E(S)$. The element $h=u^2$ is idempotent, and
\[
h=(e+f)^2=u+ef+fe\geq u.
\]
Since $h\ne u$, the four elements $e,f,u,h$ are distinct. This gives
\eqref{eq:exceptional-shape}; in particular, $h$ is the greatest additive
element. All elements other than $h$ are at most $u$, so the last
displayed equality forces $ef=h$ or $fe=h$. Interchanging the names
$e,f$ gives the stated form.
\end{proof}

\begin{lemma}\label{lem:closed-idempotents-absorption}
Suppose that every square in a four-element ai-semiring $S$ is
idempotent and that $E(S)$ is a subsemiring. Then
\eqref{eq:structural-absorption} holds.
\end{lemma}
\begin{proof}
All $s_i$ lie in $E=E(S)$. Put $a=s_1$ and
$b=s_2\cdots s_n\in E$, so $p=ab$, and write $q=pz$.
For any $v,w\in E$,
\begin{equation}\label{eq:band-order}
vw\leq(v+w)^2=v+w.
\end{equation}
Consequently, if $z\in E$, then $bz\leq b+z$ and
$q\leq p+az$. We may therefore assume that $z\notin E$ and
\begin{equation}\label{eq:anchor-failure}
q\not\leq p+az.
\end{equation}
The set $E$ is nonempty. If $|E|=1$, then $p=a$ and $q=az$, contrary
to~\eqref{eq:anchor-failure}. The case $|E|=4$ was already covered.

Suppose that $|E|=2$. Since $p\ne a$, write $E=\{a,p\}$.
Its additive order is a chain. If $p\leq a$, then $q=pz\leq az$,
so $a<p$. Also $b=p$ and $ap=p$. The element $q$ cannot lie in $E$,
whose greatest element is $p$. Nor can $q=z$, since then
$az=a(pz)=(ap)z=q$. Thus $S=\{a,p,z,q\}$, with
\[
aq=pq=q.
\]
Since $q^2\in E$, this implies $q^3=q$. Set $v=q+q^2$.
Then $v^2=v$ and $vq=v$, so $v\in E$. But every element of $E$
multiplies $q$ on the left to give $q$. Hence $v=vq=q$, contrary
to $q\notin E$.

It remains to consider $|E|=3$, so $z$ is the unique nonidempotent.
If $b+z$ were idempotent, expansion of $(b+z)^2$ would give
$bz\leq b+z$, contradicting~\eqref{eq:anchor-failure} after left
multiplication by $a$. Thus $b+z=z$ and $b\leq z$. Moreover,
$bz\not\leq z$, for otherwise $q\leq az$. Put
\[
c=bz\in E,\qquad h=z^2\in E.
\]
Monotonicity gives
\begin{equation}\label{eq:bch}
b=b^2\leq bz=c\leq z^2=h,\qquad c\not\leq z,
\end{equation}
so $b<c\leq h$.

Assume first that $c\ne h$. Then $E=\{b,c,h\}$ with $b<c<h$.
Since $h\not\leq z$, the join $z+h$ is different from $z$ and hence
belongs to $E$. It follows that $z+h=h$. Thus $h$ is the greatest
element of $S$ and $hz=zh=h$, while
\[
cz=bz^2=bh\geq bz=c.
\]
There are now three possibilities for $a$. If $a=b$, then
$q=bc=c=az$. If $a=c$, then $q=c^2=c\leq cz=az$.
If $a=h$, then $az=h$ is greatest. Each contradicts
\eqref{eq:anchor-failure}.

Therefore $c=h$, that is, $bz=h$. Since $E$ is multiplicatively closed,
\[
z^3=bz^2=bh\in E.
\]
Every square is idempotent, so $z^4=z^2=h$ and $z^6=h$.
The idempotence of $z^3$ now gives $z^3=(z^3)^2=z^6=h$. In particular,
\begin{equation}\label{eq:right-zero-part}
bh=hz=zh=h.
\end{equation}
The choices $a=b$ and $a=h$ both give $q=az=h$ and are excluded.
Hence $E=\{a,b,h\}$. Since $p=ab\in E$, the choice $p=a$ would
again give $q=az$. Thus $p\in\{b,h\}$, whence $q=pz=h$.
But $q=abz=ah$, so $ah=h$. Together with
\eqref{eq:right-zero-part}, this shows that
\begin{equation}\label{eq:right-zero}
sh=h\qquad(s\in S).
\end{equation}

Finally, $b=s_2\cdots s_n\ne a$ implies that some $s_i$, $i\geq2$,
belongs to $\{b,h\}$; otherwise that product would be $a$.
For this index, $s_i z=h$ and
\[
t_i s_i z=t_i h=h=q.
\]
Thus a summand on the right of~\eqref{eq:structural-absorption} is
exactly its left side, completing the proof.
\end{proof}

\begin{lemma}\label{lem:exceptional-idempotents-absorption}
Suppose that every square in a four-element ai-semiring $S$ is
idempotent but $E(S)$ is not a subsemiring. Then
\eqref{eq:structural-absorption} holds.
\end{lemma}
\begin{proof}
Use Lemma~\ref{lem:idempotent-closure} and its notation. Thus $h$ is
the greatest additive element, $e+f=u$, $u^2=h$, $ef=h$, and
$s_i\in\{e,f,h\}$. Since $u<h$, monotonicity gives $uh=hu=h$.
Distributivity and associativity give
\begin{equation}\label{eq:four-forms}
eu=uf=eh=hf=h,\qquad
ue=e+fe,\quad fu=fe+f,\quad he=e(fe),\quad fh=(fe)f.
\end{equation}
These formulas reduce the proof to the four possible values of $fe$.

If $fe=h$, then $h$ is a multiplicative zero. A product of elements
of $\{e,f,h\}$ is either $h$, or consists entirely of $e$'s or entirely
of $f$'s. Hence either $p=h$ or $p=s_1$, and the desired inequality
follows from $p$ or $s_1z$.

If $fe=u$, then $h$ is again a multiplicative zero, and
$ue=fu=u$, $uf=eu=h$. A product of elements of $\{e,f,h\}$ with
value $e$ or $f$ consists entirely of that element, so again
$p=s_1$. The value $p=h$ is immediate. If $p=u$, at least one factor
$s_i$ is $e$. When $z=e$, one has $pz=u=p$.
For $z\in\{f,u,h\}$, one has $pz=h$ and $s_i z=ez=h$.
Thus $t_i s_i z=h=pz$.

If $fe=e$, then $he=e$, $fh=h$, $ue=e$, and $fu=u$.
In particular $sh=h$ for every $s\in S$, and $E(S)$ is
multiplicatively closed. If $p=h$, there is nothing to prove.
If $p=f$, all factors $s_i$ are $f$, so $p=s_1$.
If $p=e$, some factor $s_i$ is $e$, since $\{f,h\}$ is a subsemigroup.
For $z=e$, one has $pz=p$; otherwise $pz=ez=h$ and
$t_i s_i z=t_i h=h$.

Finally, if $fe=f$, then $fh=f$, $he=h$, and both $f$ and $h$
are left zeros for multiplication. On $\{e,f,h\}$, a product starting
with $f$ stays $f$, one starting with $h$ stays $h$, and one starting
with $e$ either stays $e$ or changes to $h$. Thus $p=s_1$ or $p=h$,
which proves the inequality in the last case.
\end{proof}

\begin{lemma}\label{lem:nonidempotent-square-absorption}
Let $S$ be a four-element ai-semiring not isomorphic to $G_3$.
If some square in $S$ is not idempotent, then
\eqref{eq:structural-absorption} holds.
\end{lemma}
\begin{proof}
Choose $a\in S$ with $a^2\notin E(S)$ and apply
Lemma~\ref{lem:small-powers}. If
$S=\{a,a^2,a^3,a^4\}$ with $a^5=a^4$, put $h=a^4$.
Then $h$ is a multiplicative zero, every square belongs to
$\{a^2,h\}$, and the product of any two squares is $h$.
Since $n\geq2$, we have $p=h=pz$.

Otherwise put $b=a^2$, $h=a^3$, and
$N=\{a,b,h\}$. These three elements are distinct, and multiplication
on $N$ satisfies
\begin{equation}\label{eq:nil-three}
a^2=b,\qquad ab=ba=b^2=ah=ha=bh=hb=h^2=h.
\end{equation}
Write $S=N\cup\{v\}$. By Lemma~\ref{lem:small-powers}, $v^4$ is
idempotent. Therefore $v^2\ne a$, since $v^2=a$ would imply
$v^4=b\notin E(S)$. It follows that
\[
v^2\in\{b,h,v\}.
\]
We shall distinguish the case that $v$ is idempotent from the others.

Suppose first that $v^2\in\{b,h\}$. We claim that $h$ is a
multiplicative zero of $S$. Indeed,
\[
hv=b(av).
\]
If $av\in N$, this is $h$ by~\eqref{eq:nil-three}; if $av=v$,
then $bv=a(av)=v$ and $hv=a(bv)=v$. Thus $hv\in\{h,v\}$.
If $hv=v$, associativity gives $hv^2=v^2$, whereas $v^2\in N$
gives $hv^2=h$. Hence $v^2=h$, $v^3=v$, and $vh=v$.
The set $\{h,v\}$ is then a group of order two. But $h$ is the only
idempotent of $S$, contradicting the subgroup-join argument in
Lemma~\ref{lem:small-powers}. Consequently $hv=h$; the same argument
with the multiplication reversed gives $vh=h$. All squares now belong
to $\{b,h\}$, so $p=h=pz$.

Suppose next that $v^2=v$. Neither $av$ nor $va$ can be $b$.
For example, $av=b$ would imply $bv=a(av)=h$, but also
$av=(av)v=bv=h$. Thus
\[
av,va\in\{a,h,v\}.
\]
Using $(av)a=a(va)$ and~\eqref{eq:nil-three}, one obtains exactly
the following possibilities:
\begin{equation}\label{eq:extension-actions}
(av,va)=(a,a),\ (h,h),\ (h,v),\ (v,h),\ (v,v).
\end{equation}
To see that no pair has been omitted, if $av=a$, the left side of
$(av)a=a(va)$ is $b$, forcing $va=a$; the symmetric argument treats
$va=a$. When neither value is $a$, both belong to $\{h,v\}$.
The actions of $v$ on $b$ and $h$ are obtained by associativity.
In all five cases $s_i\in\{b,h,v\}$.

For $(av,va)=(a,a)$, the element $v$ is a multiplicative identity
and $h$ is a multiplicative zero. If $p=v$, then $s_1=v$ and
$pz=s_1z$; if $p=h$, then $pz=p$. The remaining possibility is $p=b$.
There is then a factor $s_i=b$ and all other factors are $v$.
If $z=v$, then $pz=p$; if $z\in N$, then
$pz=bz=h=t_i s_i z$.

For $(av,va)=(h,h)$, one has $vN=Nv=\{h\}$ and $h$ is a
multiplicative zero. Thus $p=v$ when all $s_i=v$, and $p=h$
otherwise. These alternatives are absorbed by $s_1z$ and $p$,
respectively.

For $(av,va)=(h,v)$, the element $v$ is a left zero, while
$Nv=\{h\}$ and $h$ is a left zero. If $s_1=v$, then $p=v=s_1$.
If $s_1\in\{b,h\}$, the first two factors have product $h$ and
$p=h=pz$.

For $(av,va)=(v,h)$, one has $Nv=\{v\}$ and $vN=\{h\}$.
The product $p$ belongs to $\{h,v\}$. If $z=v$, then
$pz=v=s_1z$. If $z\in N$, then every element of $\{b,h,v\}$
multiplies $z$ to give $h$, so again $pz=h=s_1z$.

For $(av,va)=(v,v)$, the element $v$ is a multiplicative zero.
If some $s_i=v$, then $p=v=pz$. Otherwise $p=h$.
For $z\in N$ one has $pz=p$, and for $z=v$ one has $pz=s_1z=v$.
This completes all cases of~\eqref{eq:extension-actions}.
\end{proof}

\begin{proof}[Proof of Proposition~\ref{prop:strong-absorption}]
If some square is not idempotent, apply
Lemma~\ref{lem:nonidempotent-square-absorption}. Otherwise every square
is idempotent. According as $E(S)$ is or is not a subsemiring,
apply Lemma~\ref{lem:closed-idempotents-absorption} or
Lemma~\ref{lem:exceptional-idempotents-absorption}.
\end{proof}

\begin{proposition}\label{prop:valid}
Every four-element ai-semiring satisfies $\sigma_n$ for every
$n\in\Ncal$.
\end{proposition}
\begin{proof}
Let $S$ be a four-element ai-semiring, fix $n\in\Ncal$, and fix an
arbitrary valuation of the variables of $\sigma_n$ in $S$.
If $S\not\cong G_3$, take
\[
t_i=y_i y_{i+1}y_{i+2}\quad(1\leq i\leq n)
\]
in Proposition~\ref{prop:strong-absorption}. Its conclusion is exactly
$q_n\leq\bfam_n$.

It remains to consider $S=G_3$. If $\bfam_n=\infty$, the inequality
is automatic. Otherwise all its summands have one common value
$c\in C_3$, and every variable has its value in $C_3$, since it occurs
in a summand. Multiplying the $n$ local summands in this group gives
\[
c^n=\prod_{i=1}^{n}\bigl(y_i y_{i+1}y_{i+2}x_i^2z\bigr)
   =\left(\prod_{j=1}^{n}y_j^3\right)
      \left(\prod_{i=1}^{n}x_i^2\right)z^n
   =p_n z^n.
\]
Here commutativity is used only inside the cyclic group $C_3$.
Since $p_n=c$, $n\equiv1\pmod3$, and the group has exponent three,
this equality becomes $c=cz$. Thus $z=1$ and
$q_n=p_nz=c=\bfam_n$. The identity therefore holds in every case.
\end{proof}

\begin{proof}[Proof of Theorem~\ref{thm:main}]
By Lemma~\ref{lem:inclusion}, the variety $\A$ contains $H,D,G_2,G_3$.
Proposition~\ref{prop:valid} and preservation of identities under
homomorphic images, subalgebras, and direct products give
$\A\models\sigma_n$ for every $n\in\Ncal$.
Theorem~\ref{thm:criterion} now implies that $\A$ is nonfinitely based.
\end{proof}

The same argument applies throughout an interval of the subvariety
lattice of $\A$.

\begin{corollary}\label{cor:interval}
Every ai-semiring variety $\mathcal W$ satisfying
\[
\T\subseteq\mathcal W\subseteq\A
\]
is nonfinitely based. In fact, no such variety has an equational basis in
a bounded number of variables.
\end{corollary}
\begin{proof}
The lower inclusion supplies the four auxiliary semirings, and the upper
inclusion supplies every $\sigma_n$. Apply Theorem~\ref{thm:criterion}.
\end{proof}

In particular, $\T$ is nonfinitely based. Its cyclic identities also have
the direct verification in Lemma~\ref{lem:validT}. Neither the application
to $\A$ nor the interval conclusion requires an equational basis for any
individual four-element generator.


\begin{thebibliography}{99}
\bibitem{BS}
S. Burris and H. P. Sankappanavar,
\emph{A Course in Universal Algebra}, Springer, New York, 1981.

\bibitem{JRZ}
M. Jackson, M. M. Ren, and X. Z. Zhao,
Nonfinitely based ai-semirings with finitely based semigroup reducts,
\emph{J. Algebra} \textbf{611} (2022), 211--245.
\href{https://doi.org/10.1016/j.jalgebra.2022.07.042}{doi:10.1016/j.jalgebra.2022.07.042}.

\bibitem{LRY}
S. Lyu, M. Ren, and M. Yue,
A new limit variety of additively idempotent semirings,
preprint, arXiv:2604.18588v2, 2026.

\bibitem{RLLC}
M. Ren, J. Liu, L. Zeng, and M. Chen,
The finite basis problem for additively idempotent semirings of order four, I,
\emph{Semigroup Forum} \textbf{110} (2025), 422--457.
\href{https://doi.org/10.1007/s00233-025-10520-7}{doi:10.1007/s00233-025-10520-7}.

\bibitem{RLYC}
M. Ren, Z. Liu, M. Yue, and Y. Chen,
The finite basis problem for additively idempotent semirings of order four, III,
\emph{Semigroup Forum} \textbf{112} (2026), 541--573.
\href{https://doi.org/10.1007/s00233-025-10584-5}{doi:10.1007/s00233-025-10584-5}.

\bibitem{SR}
Y. Shao and M. M. Ren,
On the varieties generated by ai-semirings of order two,
\emph{Semigroup Forum} \textbf{91} (2015), 171--184.
\href{https://doi.org/10.1007/s00233-014-9667-z}{doi:10.1007/s00233-014-9667-z}.

\bibitem{WJW}
A. Wang, W. Ju, and L. Wang,
The variety generated by all semirings of order three is nonfinitely based,
preprint, arXiv:2609.15703v1, 2026.

\bibitem{YR124}
M. Yue and M. Ren,
A nonfinitely based additively idempotent semiring of order four,
preprint, arXiv:2605.15493v1, 2026.

\bibitem{YRG}
M. Yue, M. Ren, and Z. Gao,
Two nonfinitely-based additively idempotent semirings of order four,
\emph{Int. J. Algebra Comput.} (2026),
\href{https://doi.org/10.1142/S0218196726500414}{doi:10.1142/S0218196726500414}.

\bibitem{YRZS}
M. Yue, M. Ren, L. Zeng, and Y. Shao,
The finite basis problem for additively idempotent semirings of order four, II,
\emph{Algebra Universalis} \textbf{86} (2025), article 33.
\href{https://doi.org/10.1007/s00012-025-00908-5}{doi:10.1007/s00012-025-00908-5}.

\bibitem{ZRCSD}
X. Z. Zhao, M. M. Ren, S. Crvenkovi\'{c}, Y. Shao, and P. \DJ api\'{c},
The variety generated by an ai-semiring of order three,
\emph{Ural Math. J.} \textbf{6}(2) (2020), 117--132.
\href{https://doi.org/10.15826/umj.2020.2.012}{doi:10.15826/umj.2020.2.012}.
\end{thebibliography}
\end{document}